\documentclass[12pt]{amsart} 
\usepackage{amsmath,amsthm,amssymb,latexsym,color}
\usepackage{hyperref}

\theoremstyle{plain}
\newtheorem{theor}{Theorem}
\theoremstyle{plain}
\newtheorem{Claim}{Claim}
\theoremstyle{remark}
\newtheorem{rem}{Remark}
\theoremstyle{plain}

\newtheorem{lemma}[theor]{Lemma}
\newtheorem*{conjecture}{Conjecture}

\def\R{{\mathbb R}}

\def\dist{{\rm d_{BM}}}

\def\Vol{{\rm Vol}}

\def\sign{{\rm sign}}
\def\Class{{\mathcal C}}
\def\Ill{{\mathcal I}}
\def\Gauss{\nu}
\def\solid{\sigma}
\def\IllumSet{\eta}
\def\pseudov{{\mathcal V}}
\def\specdir{{\bf w}}
\def\conv{{\rm conv}}
\def\Id{{\rm I}}
\def\Inter{{\rm Int\,}}
\def\Sph{{\mathbb S}}

\begin{document}

\title[Cube is the strict local maximizer ]{Cube is a strict local maximizer for\\the illumination number}

\author[Galyna Livshyts, Konstantin Tikhomirov]{Galyna Livshyts, Konstantin Tikhomirov}
\address{School of Mathematics, Georgia Institute of Technology} \email{glivshyts6@math.gatech.edu}
\address{Department of Mathematics, Princeton University}
\email{kt12@math.princeton.edu}


\maketitle

\begin{abstract}
It was conjectured by Levi, Hadwiger, Gohberg and Markus that the boundary of any convex body in
$\R^n$ can be illuminated by at most $2^n$ light sources, and, moreover,
$2^n-1$ light sources suffice unless the body is a parallelotope.
We show that if a convex body is close to the cube in the Banach--Mazur metric,
and it is not a parallelotope, then indeed $2^n-1$ light sources suffice to illuminate its boundary.
Equivalently, any convex body sufficiently close to the cube, but not isometric to it, can be covered by $2^n-1$
smaller homothetic copies of itself.
\end{abstract}

\section{Introduction}

The Levi--Hadwiger--Gohberg--Markus illumination conjecture \cite{H57,H60} is one of the famous old questions in
discrete geometry, which has several equivalent formulations. The viewpoint we adopt in this note is due to Boltyanski
(see, for example, \cite{BMS,B10,BK}):
Let $B$ be a convex body (i.e.\ compact convex set with non-empty interior) in $\R^n$
and let $\partial B$ denote its boundary.
We say that a point $x\in\partial B$ is {\it illuminated in direction $y\in\R^n\setminus\{0\}$}
if there is a small positive $\varepsilon$ such that $x+\varepsilon y$ lies in the interior of $B$.
Further, a collection $\{y^1,y^2,\dots,y^m\}$ of non-zero vectors {\it illuminates $B$}
if for any $x\in\partial B$ there is $i=i(x)\leq m$ such that $x$ is illuminated in direction $y^i$.
For any convex body $B$, denote by $\Ill(B)$ the cardinality of the smallest set of directions
sufficient to illuminate all the points from $\partial B$; it is called {\it the illumination number of $B$}.
\begin{conjecture}[Levi--Hadwiger--Gohberg--Markus]
For any convex body $B$ in $\R^n$ we have $\Ill(B)\leq 2^n$, and
the equality holds if and only if $B$ is a parallelotope.
\end{conjecture}

As of this writing, the best known general upper bound for the illumination number is due to Rogers \cite{R57}:
\begin{equation*}
\Ill(B)\leq (n\log n+n\log\log n+5n)\frac{\Vol_n(B-B)}{\Vol_n(B)},
\end{equation*}
where $B-B$ is the Minkowski sum of $B$ and $-B$ (called the difference body of $B$), and $\Vol_n(\cdot)$ is the Lebesgue volume in $\R^n$
(see, in particular, \cite[Corollary~2.11]{BK} and \cite[Corollary~3.4.2]{B10}, as well as
\cite{LT} where a ``fully random'' proof of Rogers' theorem is given).
If $B$ is origin-symmetric then $\Vol_n(B-B)=2^n\Vol_n(B)$, whence
$\Ill(B)\leq (n\log n+n\log\log n+5n)2^n$.
For non-symmetric convex bodies, the relation $\Vol_n(B-B)\leq 4^n\Vol_n(B)$ due to Rogers and Shephard \cite{RS}
implies that $\Ill(B)\leq (n\log n+n\log\log n+5n)4^n$.

The illumination conjecture has been solved (or almost solved) in some special cases.
In particular, it is known that each origin-symmetric convex body in $\R^3$ can be illuminated in at most $8$ directions
\cite{Lassak}.
The conjecture holds true for so-called {\it belt polytopes} \cite{M}
and their generalization --- {\it the belt bodies} \cite{B96};
for bodies of constant width \cite{S}
and, more generally, for {\it the fat spindle bodies} \cite{B12};
for dual cyclic polytopes \cite{BB}.
Recently, it has been shown in \cite{T} that unit balls of {\it $1$-symmetric} normed spaces
different from $\ell_\infty^n$ can be illuminated in less than $2^n$ directions.
Finally, it is of interest to note that there exist convex bodies arbitrarily close to the Euclidean ball
in the Banach--Mazur metric whose illumination number is exponential in dimension \cite{Naszodi}.
Let us refer to \cite[Chapter~3]{B10} and \cite{BK} for more information on the subject.
Let us also mention a computer assisted approach to Hadwiger's conjecture \cite{zong}.

\medskip

The approach to the Illumination problem that we consider in this note
was inspired by the work \cite{NPRZ} devoted to estimating the product of volumes
of a convex body and its polar.
In what follows, for any $p\in[1,\infty]$ by $B_p^n$ we denote the closed unit ball of the
canonical $\|\cdot\|_p$--norm in $\R^n$.
{\it The Mahler conjecture}, one of the central questions
in convex geometry, asserts that for any origin-symmetric convex body $L$, the Mahler volume
$\Vol_n(L)\cdot \Vol_n(L^\circ)$ is greater or equal to the Mahler volume of the cube $\Vol_n(B_\infty^n)\cdot\Vol_n(B_1^n)$,
where $L^\circ$ denotes the polar body for $L$.
We refer to \cite{BM} (see also \cite{K}) for an ``isomorphic solution'' to this problem and related information. Very recently, this conjecture was verified in $\R^3$ in \cite{mahler}. In \cite{NPRZ} the Mahler conjecture was
confirmed in every dimension in a small neighborhood of the cube and, moreover,
it was shown that the cube is a {\it strict} local minimizer in the Banach--Mazur metric on the class of symmetric convex bodies.
Let us recall that for any two (not necessarily centrally-symmetric) convex bodies $B$ and $L$
in $\R^n$, {\it the Banach--Mazur distance} $\dist(B,L)$ between $B$ and $L$ is defined as
the infimum of $\lambda\geq 1$ such that there is an invertible linear operator $T_\lambda:\R^n\to\R^n$
and two vectors $x_\lambda$ and $y_\lambda$ satisfying $B\subset T_\lambda(L)+x_\lambda\subset \lambda B+y_\lambda$.
The main theorem of \cite{NPRZ} asserts that for every positive integer $n$ there is $\delta(n)>0$ such that any
origin-symmetric convex body $L$ with $1\neq \dist(L,B_\infty^n)\leq 1+\delta(n)$ satisfies
$\Vol_n(L)\cdot \Vol_n(L^\circ)>\Vol_n(B_\infty^n)\cdot\Vol_n(B_1^n)$.

\medskip

In this paper, we apply the ``local'' viewpoint of \cite{NPRZ}
in the context of the Illumination conjecture.
The nature of the Illumination problem is very different from the Mahler conjecture discussed above;
the most obvious distinction coming from the lack of continuity.
Whereas the volume is stable with respect to small perturbations of a convex body, the structure of
its boundary --- the layout of the extreme points, the combinatorial structure (in case of polytopes) ---
can change significantly even with infinitely small perturbations. It becomes an interesting feature
that certain ``discrete--geometric'' properties related to the illumination
(which we will mention later), remain stable in a small neighborhood of the cube.
The main non-technical result of this paper is the following theorem.
\begin{theor}\label{t: main}
For any $n\geq 3$ there is a $\delta=\delta(n)>0$ with the following property:
Let $B$ be a convex body in $\R^n$ and assume that $1\neq\dist(B,B_\infty^n)\leq 1+\delta$.
Then $B$ can be illuminated with at most $2^n-1$ directions.
\end{theor}

It is well known that the illumination number of $B$ can be equivalently defined as the least number
of translates of the interior of $B$ needed to cover $B$ (see, in particular \cite{BMS}, \cite[Chapter~3]{B10}, \cite{BK}).
In this sense, it becomes a simple observation that any convex body with a sufficiently
small Banach--Mazur distance to the cube can be illuminated in at most $2^n$ directions
(in fact, more general statements are known; see \cite[Section~4]{BK} and references therein).
The non-triviality of the above theorem consists in proving that {\it strictly less}
than $2^n$ light sources suffice.

The construction of an illuminating set for $B$ involves a careful study of
its geometry. Naturally, we consider the canonical illuminating set of the cube --- the set of all sign vectors ---
as the starting point. Next, we determine which pair of adjacent illuminating directions can be ``glued together''
to form a single light source. The procedure is completed by
repositioning several of the ``canonical'' light sources in a special way.
Thus, we show that for any body $B$ close to the cube in the Banach--Mazur metric
but not isometric to the cube,
we can find a {\it distinguished} pair of boundary points of $B$ illuminated by the same light source,
and complete the illumination of the entire boundary of $B$ by adding $2^n-2$ light sources in ``standard''
or ``almost standard'' positions. Interestingly, existence of this distinguished pair
is a feature of all convex sets sufficiently close to the cube (but not cube itself).
We will give a complete description of this strategy later.

The dependency of the quantity $\delta=\delta(n)$ of the main theorem on the dimension $n$ is not explicit
as we use continuity arguments to establish certain properties of convex sets close to the cube.
In principal, all the arguments can be reworked to give $\delta$ as an explicit function of $n$,
but we prefer to avoid this for the sake of clarity.

\begin{rem}
We would like to point it out that our estimate is sharp in a sense that
for every $n\geq 2$ and for every $\varepsilon\in(0,1)$ there exists a convex body $B$ in $\R^n$
with $\dist(B,B_{\infty}^n)\leq 1+\varepsilon$ and such that the illumination number of $B$ is exactly $2^n-1$.
We consider the following construction.

Let $v:=(1,1,...,1)$, $v':=(-1,1,1,...,1)$ and, for $\varepsilon\in(0,1)$, let 
$\tilde{v}:=(1-\varepsilon,1,1,...,1)$.
Set
$$B:=\conv\{\{-1,1\}^n\setminus v, \tilde{v}\}.$$ 
Note that $\dist(B, B_{\infty}^n)\leq \frac{1}{1-\varepsilon}$. 

For each vertex $w$ of $B$, consider the set of vectors which illuminate $w$ as a boundary point of $B$:
$$\eta(B,w):=\big\{x\in \R^n: \exists \delta>0 \,s.t.\, w+\delta x\mbox{ is in the interior of }B\big\}.$$
Note that for all the vertices $w$ of $B$ which have at least two negative coordinates,
as well as for $w=v'$, one has
$\eta(B,w)=\{x\in\R^n: \sign(x_i)=-\sign(w_i)\mbox{ for all }i\leq n\}$.
Moreover, for all vertices $w$ of $B$ with exactly one negative coordinate one has
$\eta(B,w)\subset\{x\in\R^n: \sign(x_i)=-\sign(w_i)\mbox{ for all }i\leq n\}$.
This implies that the illuminating sets for $\{-1,1\}^n\setminus v$ are pairwise disjoint,
whence $\Ill(B)\geq 2^n-1$.
\end{rem}

\textbf{Aknowledgement.} The first author is supported by an AMS-Simons travel Grant.
The second author is partially supported by the Simons Foundation.
The work was partially supported by the National Science Foundation under Grant No. DMS-1440140
and the Viterbi postdoctoral fellowship
while the authors were in residence at the Mathematical Sciences Research Institute in Berkeley,
California, during the Fall 2017 semester.

\section{Notation and preliminaries}

Given a positive integer number $n$, $[n]$ is the set $\{1,2,\dots,n\}$.
We denote by $e_1,e_2,\dots,e_n$ the standard basis in $\R^n$ and by $\langle\cdot,\cdot\rangle$ ---
the canonical inner product in $\R^n$.
Given any $1\leq p\leq\infty$, let $\|\cdot\|_p$ be the $\ell_p^n$--norm in $\R^n$, i.e.\
$$\big\|(x_1,x_2,\dots,x_n)\big\|_p:=\Big(\sum_{i=1}^n |x_i|^p\Big)^{1/p}\;\;\mbox{ and }\;\;
\big\|(x_1,x_2,\dots,x_n)\big\|_\infty:=\max\limits_{i\leq n}|x_i|.$$
The unit ball of the $\ell_p^n$--norm is denoted by $B_p^n$.
By $\Id$ we denote the identity operator in $\R^n$ (the dimension $n$ will always be clear from the context).
Further, given a linear operator $T:\R^n\to\R^n$, let $\|T\|:=\|T\|_{2\to 2}$ stand for the spectral norm of $T$
(i.e.\ its largest singular value), and, more generally, for any two numbers $1\leq p,q\leq\infty$, let $\|T\|_{p\to q}$
be the operator norm of $T$ considered as a mapping from $\ell_p^n$ to $\ell_q^n$. Thus,
$$\|T\|_{p\to q}:=\sup\limits_{\|x\|_p=1}\|Tx\|_q.$$
In view of standard comparison inequalities for $\ell_p^n$--norms, we have
\begin{equation}\label{eq:operator norm comparison}
\|T\|_{p\to q}\leq n^2\|T\|_{p'\to q'}\quad\quad\mbox{for all }\;1\leq p,q,p',q'\leq \infty.
\end{equation}
Given a cone $K\subset\R^n$ with the vertex at the origin, we define {\it the solid angle} $\solid(K)$ as
$$\solid(K):=\frac{\Vol_n(K\cap B_2^n)}{\Vol_n(B_2^n)}.$$
Let $B$ be a convex body in $\R^n$. Recall that the Gauss map $\nu_B:\partial B\rightarrow \Sph^{n-1}$
maps each point $x\in \partial B$ to the collection of outer unit normals to supporting hyperplanes at $x$.
{\it The Gauss image} $\Gauss(B,x)$ of $x\in\partial B$ is the convex cone given by $\{\lambda y:\;y\in\nu_B(x),\;\lambda\geq 0\}$.
Further, each point $x\in\partial B$ can be associated with another convex cone,
{\it the illuminating set} $\IllumSet(B,x)$ which comprises all non-zero directions illuminating $x$.
Note that $\Gauss(B,x)$ and the closure of $\IllumSet(B,x)$ are {\it polar} cones. Given a point $x\in\partial B$,
denote by $\solid(B,x)$ the solid angle of the cone $\IllumSet(B,x)$.

\medskip

The next simple lemma will be useful; we give a proof for reader's convenience.
\begin{lemma}\label{l: perturbation of cube}
For any $n>1$ and $\beta>0$ there is $r_{\text{\tiny\ref{l: perturbation of cube}}}=r_{\text{\tiny\ref{l: perturbation of cube}}}(n,\beta)>0$
with the following property:
Let $P$ be a non-degenerate parallelotope in $\R^n$ such that for each vertex $v$ of the standard cube $B_\infty^n$
there is a vertex $v'$ of $P$ satisfying $v'-v\in r_{\text{\tiny\ref{l: perturbation of cube}}} B_\infty^n$.
Then there is an invertible linear operator $T$ in $\R^n$ and a vector $y$ such that
$B_\infty^n=T(P)+y$ and $\|T-\Id\|,\|T^{-1}-\Id\|\leq\beta$.
\end{lemma}
\begin{proof}
Note that, as we allow $r_{\text{\tiny\ref{l: perturbation of cube}}}$ to depend on $n$ and in view of \eqref{eq:operator norm comparison},
it is sufficient to prove the statement with the spectral norm replaced by $\|\cdot\|_{\infty\to\infty}$.
Fix a small $\beta>0$ and define $r=r_{\text{\tiny\ref{l: perturbation of cube}}}:=\frac{\beta}{4}$.
Let $P$ be a parallelotope in $\R^n$ satisfying conditions of the lemma.
First, observe that $(1-r)B_\infty^n\subset P$. Indeed, otherwise there would exist a vertex $v$ of $(1-r)B_\infty^n$
and an affine hyperplane $H$ passing through $v$ and not intersecting $P$. On the other hand,
one could always find a pair of opposite vertices of $B_\infty^n$ lying in different half-spaces (determined by $H$)
and such that the {\it $\ell_\infty^n$--distance} of either vertex to $H$ is strictly greater than $r$. This would contradict the assumption
that every vertex of the cube can be $r$-approximated by a point in $P$ in the $\ell_\infty^n$--metric.

For each $v\in\{-1,1\}^n$
let $f(v)$ be the (unique) vertex of $P$ satisfying $\|v-f(v)\|_\infty\leq r$.
Let $S$ be the set of all vertices of $B_\infty^n$ adjacent to $(1,1,\dots,1)$.
Note that $n$ vectors $\{f(1,1,\dots,1)-f(v)\}_{v\in S}$ are linearly independent whence there is a unique linear operator
$T$ and a vector $y\in\R^n$ such that $(1,1,\dots,1)=T(f(1,1,\dots,1))+y$ and $v=T(f(v))+y$
($v\in S$). Note that necessarily $T(P)+y=B_\infty^n$, and, moreover, $f$ is the restriction of $T^{-1}(\cdot)-T^{-1}(y)$
to $\{-1,1\}^n$. Together with the inclusion $(1-r)B_\infty^n\subset P$,
this gives $(1-r)T(B_\infty^n)\subset B_\infty^n-y$, whence, by the symmetry of $T(B_\infty^n)$,
$(1-r)T(B_\infty^n)\subset B_\infty^n$, and $\|T\|_{\infty\to\infty}\leq (1-r)^{-1}$.

By linearity of $T$, we have
$$0=2^{-n}\sum_{v\in\{-1,1\}^n} T^{-1} (v)=T^{-1}(y)+2^{-n}\sum_{v\in\{-1,1\}^n}f(v),$$ 
whence
\begin{align*}
\|T^{-1}(y)\|_{\infty}&=2^{-n}\Big\|\sum_{v\in \{-1,1\}^n}f(v)\Big\|_{\infty}
=2^{-n}\Big\|\sum_{v\in \{-1,1\}^n}(v-f(v))\Big\|_{\infty}\\
&\leq 2^{-n}\sum_{v\in\{-1,1\}^n} \|v-f(v)\|_{\infty}
\leq r.
\end{align*}
Next, elementary convexity properties and the bound $\|T\|_{\infty\to\infty}\leq (1-r)^{-1}<2$ imply
$$
\|T-\Id\|_{\infty\to\infty}
=\max_{v\in\{-1,1\}^n}\|Tv-v\|_{\infty}
< 
2\max_{v\in\{-1,1\}^n}\|v-f(v)-T^{-1}(y)\|_{\infty}
\leq 4r=\beta.
$$
Finally,
$$\|T^{-1}-\Id\|_{\infty\to\infty}=\max_{v\in\{-1,1\}^n}\|T^{-1} v-v\|_{\infty}
=
\max_{v\in\{-1,1\}^n}\|f(v)-v+T^{-1}(y)\|_{\infty}\leq 2r\leq \beta.$$
\end{proof}

\section{High-level structure of the proof}

In this section, we give the proof of the main theorem assuming several properties
of convex bodies close to the cube (they are stated as lemmas). The proofs of the lemmas
which constitute the technical part of the paper, are deferred to the next section.
In the proof of the theorem, we work with quantities depending on various parameters or other functions.
For example, by writing $\beta=\beta(n,\alpha)$ we introduce $\beta$ as a function of two variables $n$ and $\alpha$.
To make referencing easier, each function introduced within a lemma is written with the number of that lemma as a subscript.

\bigskip

The invariance of the illumination number under affine transformations allows us to restrict our analysis
to the class of convex bodies $B\subset\R^n$ such that
\begin{equation}\label{eq: star position}\tag{$\star$}
B_\infty^n\subset B\subset \dist(B,B_\infty^n)B_\infty^n+y
\end{equation}
for some vector $y=y(B)\in\R^n$.
We say that a body $B$ satisfying \eqref{eq: star position} is in {\it a $\star$-position}.
Note that the $\star$-position is not uniquely defined in general.
It is obvious that any convex body $B$ in a $\star$-position satisfies
\begin{equation}\label{eq: cube incl ctr}
B\subset (2\dist(B,B_\infty^n)-1)B_\infty^n.
\end{equation}

\medskip

Now, assume we have a convex body $B$ in a $\star$-position,
with a very small Banach--Mazur distance to the cube (but not the cube itself).
How could we construct an illuminating set for $B$ of cardinality $2^n-1$?
It is natural to start with the standard illuminating set for the cube, i.e.\ the set $\{-1,1\}^n$, transform it in some way and remove one direction.
One may note that simply excluding one
direction from $\{-1,1\}^n$, without changing the remaining directions is not sufficient.
Indeed, consider a convex polygon $P$ in $\R^2$ with eight
vertices $\pm (1+\varepsilon,1-\varepsilon),\pm (1-\varepsilon,1+\varepsilon),
\pm (-1+\varepsilon,1+\varepsilon),\pm (-1-\varepsilon,1-\varepsilon)$, where $\varepsilon>0$
is small enough. It can be checked that $P$ is in a $\star$-position,
and that any proper subset of $\{-1,1\}^2$ does not illuminate $P$.

Another natural approach is to replace a pair of adjacent illumination directions
with a single vector. That is, given a convex body $B$ in $\R^n$ in a $\star$-position and
with a small distance to cube, construct an illumination set of the form
\begin{equation}\label{eq:first approach}
(\{-1,1\}^n\setminus\{-v,-v'\})\,\cup\,\{\specdir\},
\end{equation}
where $\{-v,-v'\}$ is some pair of adjacent directions from $\{-1,1\}^n$ and $\specdir$ is some non-zero vector in $\R^n$.
It can be checked directly that this method fails for the polygon $P$ constructed above, but,
by slightly repositioning the canonical illumination directions, it is possible to get an illuminating set for $P$
(say, take the set $\{(1,-1.1),(1,1.1)\}\,\cup\{\specdir\}$ with $\specdir=(-1,0)$).
It turns out that this approach can be generalized to convex bodies in higher dimensions;
thus, our construction in Theorem~\ref{t: main} resembles \eqref{eq:first approach}
but is somewhat more technical.
Let us give a formal definition.

\medskip

For any $n\geq 3$ and any choice of parameters $\varepsilon,\theta\in(0,1)$, let us denote by
$\Class_n(\varepsilon,\theta)$
the collection of all subsets $S\subset\R^n$ of cardinality $2^n-1$ such that
there is a pair $v=v(S),v'=v'(S)$ of adjacent vertices of the standard cube $B_\infty^n$ satisfying
the following three conditions:
\begin{itemize}

\item For any vertex $w\in\{-1,1\}^n$ adjacent neither to $v$ nor to $v'$,
$S$ contains a vector in the set $-w+\varepsilon B_\infty^n$.

\item For any vertex $w=(w_1,w_2,\dots,w_n)\in\{-1,1\}^n\setminus\{v,v'\}$ {\it adjacent} either to $v$ or $v'$,
$S$ contains a vector in
$$-\sum_{j:j\neq i}w_j e_j-\theta w_ie_i+\varepsilon B_\infty^n,$$
where $i=i(w)$ is the unique index in $[n]$ such that
$v_i=v_i'\neq w_i$.

\item $S$ contains a vector $\specdir=(\specdir_1,\specdir_2,\dots,\specdir_n)$ in the parallelepiped
$$-\prod_{j=1}^n \bigg(\frac{v_j+v_j'}{2}\cdot [\theta,1]\bigg)+\varepsilon B_\infty^n.$$
Note here that if $\ell\in[n]$ is the unique index with $v_\ell=-v_\ell'$, we get $|\specdir_\ell|\leq\varepsilon$.

\end{itemize}

We will call $\specdir$ {\it the distinguished direction} of the set $S$,
and $v$, $v'$ the {\it distinguished vertices} of the cube w.r.t.\ set $S$.
Roughly speaking, each set in $\Class_n(\varepsilon,\theta)$ is constructed by taking
a standard illuminating set $\{-1,1\}^n$ of $B_\infty^n$, glueing together a pair of adjacent illuminating
directions and then perturbing the collection in a special way.
The principal difference of the above construction from \eqref{eq:first approach}
is that we reposition {\it all} illuminating directions adjacent to either $-v$ or $-v'$ by moving them ``closer''
to $-v$, $-v'$ (the reader may wish to compare this strategy to our illumination of the polytope $P$ in the above example).
The rest of the illuminating directions, disregarding a small perturbation, remain unchanged
(i.e.\ are essentially sign vectors).

Observe that for any $0<\varepsilon<\varepsilon'$ and $\theta>0$ we have
$\Class_n(\varepsilon,\theta)\subset \Class_n(\varepsilon',\theta)$.
The definition of the class $\Class_n(\varepsilon,\theta)$, being somewhat technical,
is designed to be ``stable'' with respect to linear transformations close to the identity:
\begin{lemma}\label{l: stability}
For any $n>2$, $\varepsilon,\theta\in(0,1)$ and $\alpha>0$
there is $\beta_{\text{\tiny\ref{l: stability}}}=\beta_{\text{\tiny\ref{l: stability}}}(n,\alpha)>0$ with the following property:
whenever $S\in\Class_n(\varepsilon,\theta)$ and $T$ is a linear operator in $\R^n$ satisfying $\|T-\Id\|\leq \beta_{\text{\tiny\ref{l: stability}}}$,
we have $T(S)\in \Class_n(\varepsilon+\alpha,\theta)$.
\end{lemma}

Our goal is to show that if $B\subset\R^n$ is a convex body in a $\star$-position which is very close to the cube, yet distinct from it, then $B$ can be illuminated by a collection
$S\in\Class_n(\varepsilon,\theta)$ for some appropriately chosen parameters $\varepsilon,\theta$.
Both parameters will be taken sufficiently small, but, importantly,
$\varepsilon$ shall be much smaller than $\theta$.

\bigskip

A crucial notion that will help us to study illumination by directions from $\Class_n(\varepsilon,\theta)$
is that of {\it pseudo-vertices}.
Let $r\in(0,1)$ be a parameter and let $B$ be a body in
$\R^n$ in a $\star$-position with $\dist(B,B_\infty^n)\leq 1+\frac{r}{2}$.
Note that, in view of \eqref{eq: cube incl ctr}, for any vertex $v$ of the standard cube $B_\infty^n$,
the set $r B_\infty^n+v$ intersects with the boundary of $B$.
Given any vertex $v$ of $B_\infty^n$, we say that a point $p\in (r B_\infty^n+v)\cap \partial B$
is {\it an $(r)$-pseudo-vertex of $B$} if $p$ has minimal solid angle $\solid(B,p)$
among all points in $(r B_\infty^n+v)\cap \partial B$.
For all admissible $r$, the set $B$ has {\it at least} $2^n$ $(r)$-pseudo-vertices, but may have
more (in fact, uncountably many if the boundary of $B$ is smooth).
Further, we say that a collection of $2^n$ points $\pseudov\subset\R^n$ is a {\it proper set of $(r)$-pseudo-vertices} of $B$ if
for any vertex $v$ of $B_\infty^n$, $\pseudov$ contains exactly one $(r)$-pseudo-vertex $p(v)\in (r B_\infty^n+v)\cap \partial B$.
In particular, if $B=B_\infty^n$ then the proper set of $(r)$-pseudo-vertices (for any admissible value of $r$)
coincides with the set of the regular vertices of the cube.
In general, $\pseudov$ may be not uniquely defined.

\medskip

It is easy to see that any point $x$ on the boundary of the standard cube such that
$\solid(B_\infty^n,x)<2^{-n+1}$, must be one of its vertices.
Below we state a weaker relative of this property
for convex bodies sufficiently close to the cube and their pseudo-vertices.

\begin{lemma}\label{l: pseudov}
For each $n>2$ there are $\eta_{\text{\tiny\ref{l: pseudov}}}=\eta_{\text{\tiny\ref{l: pseudov}}}(n)\in(0,1)$,
$\kappa_{\text{\tiny\ref{l: pseudov}}}=\kappa_{\text{\tiny\ref{l: pseudov}}}(n)>0$ with the following property.
Let $0<\eta\leq \eta_{\text{\tiny\ref{l: pseudov}}}$.
Then there is $\delta_{\text{\tiny\ref{l: pseudov}}}=\delta_{\text{\tiny\ref{l: pseudov}}}(n,\eta)\in(0,\eta/2)$
such that for any convex body $B$ in $\R^n$ in the $\star$-position, with $\dist(B,B_{\infty}^n)\leq 1+\delta_{\text{\tiny\ref{l: pseudov}}}$,
and for any point $x\in\partial B$ with $\solid(B,x)\leq (1+\kappa_{\text{\tiny\ref{l: pseudov}}})\cdot 2^{-n}$,
there is a vertex $v$ of the standard cube $B_\infty^n$ such that
$x$ is the unique $(\eta)$-pseudo-vertex of $B$
in $(\eta B_\infty^n+v)\cap \partial B$.
\end{lemma}
Thus, we can detect a pseudo-vertex if its solid angle is less than certain critical value.
The following lemma provides a connection between this property of a boundary point and
the illumination of $B$, and, together with Lemma~\ref{l: pseudo not cubic}, comprises the most technical part of the paper:
\begin{lemma}\label{l: solid}
For any $n>2$, $\kappa>0$ there is $\theta_{\text{\tiny\ref{l: solid}}}(n,\kappa)\in(0,1)$
with the following property. For any $0<\theta\leq \theta_{\text{\tiny\ref{l: solid}}}$ there are
$\delta_{\text{\tiny\ref{l: solid}}}=\delta_{\text{\tiny\ref{l: solid}}}(n,\kappa,\theta)>0$
and $\varepsilon_{\text{\tiny\ref{l: solid}}}=\varepsilon_{\text{\tiny\ref{l: solid}}}(n,\kappa,\theta)\in(0,1)$
such that for any convex body $B$ in a $\star$-position with $\dist(B,B_\infty^n)\leq 1+\delta_{\text{\tiny\ref{l: solid}}}$
and any element $S\in\Class_n(\varepsilon_{\text{\tiny\ref{l: solid}}},\theta)$ we have that
every point $x\in\partial B$ not illuminated by $S$ satisfies
$$\solid(B,x)\leq (1+\kappa)\cdot 2^{-n}.$$
\end{lemma}

\medskip

The following two lemmas are the core of our argument:
\begin{lemma}\label{l: pseudo=cubic}
Let $n>2$ and let $\varepsilon,\theta\in(0,1)$ be any numbers.
Let $B$ be a convex body in $\R^n$ such that $\{-1,1\}^n\subset \partial B$ and $B\neq B_\infty^n$.
Then there is a set $S\in\Class_n(\varepsilon,\theta)$
which illuminates every point in $\{-1,1\}^n$ (viewed as boundary points for $B$).
\end{lemma}
\begin{rem}
Note that in the above lemma we illuminate only a subset of the boundary and not the entire body.
\end{rem}

\begin{lemma}\label{l: pseudo not cubic}
For any $n>2$ and $\varepsilon\in(0,1)$ there is
$\theta_{\text{\tiny\ref{l: pseudo not cubic}}}=\theta_{\text{\tiny\ref{l: pseudo not cubic}}}(n)$
with the following property.
Let $0<\theta\leq \theta_{\text{\tiny\ref{l: pseudo not cubic}}}$. Then there is
$r_{\text{\tiny\ref{l: pseudo not cubic}}}=r_{\text{\tiny\ref{l: pseudo not cubic}}}(n,\varepsilon,\theta)\in(0,1/2)$
such that for any convex polytope $P$ with $2^n$ vertices and $\dist(P,B_\infty^n)\neq 1$ so that
for each $v\in\{-1,1\}^n$ there is a vertex $v'$ of $P$ with $v'-v\in r_{\text{\tiny\ref{l: pseudo not cubic}}}B_\infty^n$,
we have that $P$ is illuminated by some set $S\in \Class_n(\varepsilon,\theta)$.
\end{lemma}
\begin{rem}
It is crucial that parameter $\theta_{\text{\tiny\ref{l: pseudo not cubic}}}$
in the above lemma depends only on $n$ and {\it not} on $\varepsilon$.
\end{rem}

\medskip

The above statements allow to complete the proof of the main theorem:

\begin{proof}[Proof of Theorem~\ref{t: main}]

Let us start by defining parameters. We fix any $n>2$ and set
\begin{align*}
&\theta:=\min(\theta_{\text{\tiny\ref{l: pseudo not cubic}}}(n),\theta_{\text{\tiny\ref{l: solid}}}(n,\kappa_{\text{\tiny\ref{l: pseudov}}}(n))),\;\;
\varepsilon:=\varepsilon_{\text{\tiny\ref{l: solid}}}(n,\kappa_{\text{\tiny\ref{l: pseudov}}}(n),\theta),\;\;
\beta:=\beta_{\text{\tiny\ref{l: stability}}}(n,\varepsilon/2),\\
&r:=\min\big(r_{\text{\tiny\ref{l: pseudo not cubic}}}(n,\varepsilon,\theta),r_{\text{\tiny\ref{l: perturbation of cube}}}(n,\beta),\eta_{\text{\tiny\ref{l: pseudov}}}(n)\big),\;\;
\delta:=\min\big(\delta_{\text{\tiny\ref{l: solid}}}(n,\kappa_{\text{\tiny\ref{l: pseudov}}}(n),\theta),\delta_{\text{\tiny\ref{l: pseudov}}}(n,r)\big).
\end{align*}
Consider a convex body $B$ in $\R^n$ with $1\neq\dist(B,B_\infty^n)\leq 1+\delta$.
Assume that $B$ is in a $\star$-position.
Let $\pseudov$ be a proper set of $(r)$-pseudo-vertices of $B$ (note that $\delta\leq r/2$ so $\pseudov$ exists).

As the first step, we show that there is a set of directions $S\in\Class_n(\varepsilon,\theta)$
which illuminates $\conv(\pseudov)$.
Indeed, if $\conv(\pseudov)$ is not a parallelotope then the assertion follows from Lemma~\ref{l: pseudo not cubic}
and our choice of parameters.
Otherwise, if $\dist(\conv(\pseudov),B_\infty^n)=1$ then, in view of Lemma~\ref{l: perturbation of cube},
there is a linear operator $T$ and a vector $y$ in $\R^n$ such that
$\|T-\Id\|\leq\beta$, $\|T^{-1}-\Id\|\leq\beta$ and $T(\conv(\pseudov))+y=B_\infty^n$.
By Lemma~\ref{l: pseudo=cubic}, we can find a set $S'\in\Class_n(\varepsilon/2,\theta)$
which illuminates points in $\{-1,1\}^n$ considered as boundary points of $T(B)+y$.
Hence, $T^{-1}(S')$ illuminates $\conv(\pseudov)$ (again, viewed as a subset of the boundary of $B$).
Due to the assumptions on $T$ and Lemma~\ref{l: stability}, $S:=T^{-1}(S')$ belongs to $\Class_n(\varepsilon,\theta)$.

Now, having constructed $S$, assume that there is a point $x\in\partial B$ which is not illuminated by $S$.
Then, in view of Lemma~\ref{l: solid} and our choice of parameters,
we have $\solid(B,x)\leq (1+\kappa_{\text{\tiny\ref{l: pseudov}}}(n))\cdot 2^{-n}$.
But then, by Lemma~\ref{l: pseudov} (applied with $\eta:=r$), $x$ is the unique $(r)$-pseudo-vertex
in $(r B_\infty^n+v)\cap \partial B$ for some vertex $v$ of the standard cube $B_\infty^n$. Hence, $x$ must belong to $\pseudov$
leading to contradiction.
Thus, $B$ is entirely illuminated in $2^n-1$ directions.
\end{proof}

\section{{Proofs of Lemmas~\ref{l: stability}---\ref{l: pseudo not cubic}}}

\subsection{Proof of Lemma~\ref{l: stability}.}
Fix $n>2$ and parameters $\varepsilon,\theta\in(0,1)$, $\alpha>0$.
Instead of the spectral norm $\|\cdot\|$, it will be convenient to consider $\|\cdot\|_{\infty\to\infty}$
(this makes no difference since we allow the parameter $\beta_{\text{\tiny\ref{l: stability}}}$ to depend on $n$). 
Set $\beta:=\alpha/2$.
Suppose $T$ is a linear operator satisfying $\|T-\Id\|_{\infty\to\infty}\leq \beta$. 

Take a vertex $w$ of the standard cube and let $p\in -w+\varepsilon B_{\infty}^n$.
Observe that $\|p+w\|_{\infty}\leq \varepsilon$
and $\|p\|_{\infty}\leq 1+\varepsilon$. Thus
$$\|Tp+w\|_{\infty}\leq \|Tp-p\|_{\infty}+\|p+w\|_{\infty}
\leq \|T-\Id\|_{\infty\to\infty}\|p\|_{\infty}+\|p+w\|_{\infty}\leq \beta(1+\varepsilon)+\varepsilon\leq \alpha+\varepsilon,$$
and therefore $Tp\in -w+(\alpha+\varepsilon)B_{\infty}^n$.

In the same manner, given any $i\leq n$ and any $p'\in -\sum_{j:j\neq i}w_j e_j-\theta w_ie_i+\varepsilon B_\infty^n$,
we have
$$\Big\|Tp'+\Big(\sum_{j:j\neq i}w_j e_j+\theta w_ie_i\Big)\Big\|_{\infty}\leq \alpha+\varepsilon.$$
Finally, for any two adjacent vertices $v,v'$ of the standard cube and for any
$$p''\in -\prod_{j=1}^n\left(\frac{v_j+v_j'}{2}\cdot[\theta,1]\right)+\varepsilon B_\infty^n,$$
we have
$$Tp''\in -\prod_{j=1}^n\left(\frac{v_j+v_j'}{2}\cdot[\theta,1]\right)+(\alpha+\varepsilon) B_\infty^n.$$
Together with the definition of the classes $\Class_n(\varepsilon,\theta)$, this implies the result.

\subsection{Proof of Lemma~\ref{l: pseudov}}

We will prove the following two claims.

\begin{Claim}\label{claim1}
For any $n>2$ and $\eta>0$ there is $\delta_0(n,\eta)>0$ with the following property.
Let $B$ be a convex body in $\R^n$ with $B_\infty^n\subset B\subset (1+\delta_0)B_\infty^n$,
and let $x$ be a point on the boundary of $B$ such that $\|x-v\|_\infty\geq \eta$ for all $v\in\{-1,1\}^n$.
Then necessarily $\sigma(B,x)\geq 1.5\cdot 2^{-n}$.
\end{Claim}

\begin{Claim}\label{claim2}
For any $n>2$ there are positive $\eta'=\eta'(n)<1$ and $\kappa'=\kappa'(n)<0.5$ with the following property.
Let $B$ be a convex body in $\R^n$ satisfying $B_\infty^n\subset B$, let $v$ be any vertex of $B_\infty^n$
and let $x,y$ be two distinct points in $\partial B\cap(v+\eta' B_\infty^n)$. Then necessarily $\max(\solid(x),\solid(y))> (1+\kappa')2^{-n}$.
\end{Claim}

Now, it is not difficult to verify that the claims imply the assertion of the lemma.
Indeed, fix any $0<\eta\leq \eta_{\text{\tiny\ref{l: pseudov}}}:=\eta'(n)$ and set
$\delta_{\text{\tiny\ref{l: pseudov}}}:=\min(\delta_0(n,\eta),\eta)/2$.
Let $B$ be a convex body in $\R^n$ in a $\star$--position, with $\dist(B,B_\infty^n)\leq 1+\delta_{\text{\tiny\ref{l: pseudov}}}$.
Observe that, in view of \eqref{eq: cube incl ctr}, we have $B\subset (1+2\delta_{\text{\tiny\ref{l: pseudov}}})B_\infty^n$.
Let $x$ be a point on the boundary of $B$ such that $\solid(B,x)\leq (1+\kappa_{\text{\tiny\ref{l: pseudov}}})2^{-n}$,
with $\kappa_{\text{\tiny\ref{l: pseudov}}}:=\kappa'(n)$ In view of Claim~\ref{claim1},
we have $\|x-v\|_\infty\leq \eta$
for some $v\in\{-1,1\}^n$. At the same time, for any other point $y\in\partial B\cap (v+\eta B_\infty^n)$
we have, in view of Claim~\ref{claim2}, $\solid(B,y)> (1+\kappa_{\text{\tiny\ref{l: pseudov}}})2^{-n}\geq \solid(B,x)$.
In other words, $x$ is the (unique) minimizer for $\solid(B,\cdot)$ in $\partial B\cap (v+\eta B_\infty^n)$.
The lemma follows. Now, we prove the claims.

\begin{proof}[Proof of Claim~\ref{claim1}]
For any point $x\in \R^n\setminus \Inter(B_{\infty}^n)$, define a convex cone
$K_x:=\{x+t z:\,z\in B_{\infty}^n,\, t\geq 0\}$, and define a function
$f(x): \R^n\setminus \Inter(B_{\infty}^n)\rightarrow \R$ by
$f(x):=\sigma(K_x, x)$. Observe that $f(x)$ is lower semi-continuous,
that is, for any sequence $(x^m)_{m\geq 1}$ in $\R^n\setminus \Inter(B_{\infty}^n)$ converging to a point $x$ we have
$
f(x)\leq \liminf_{m\rightarrow \infty} f(x^m)
$.
Now, for any $\delta>0$ and $0<\eta<1$ consider the set 
$$A(\eta, \delta):=\big((1+\delta) B_{\infty}^n\setminus \Inter(B_{\infty}^n)\big)
\cap\big\{y\in\R^n: \|y-v\|_{\infty}\geq \eta\;\mbox{for all}\; v\in\{-1,1\}^n \big\}.$$
In words, $A(\eta, \delta)$ is the set of points in the closed thin shell between the boundaries of cubes
$B_\infty^n$ and $(1+\delta)B_\infty^n$, with $\ell_\infty^n$--distance to $\{-1,1\}^n$ at least $\eta$.
Clearly, for any fixed $0<\eta<1$ the lower semi-continuity of $f(x)$ implies that
the limit $\lim\limits_{\delta\to 0}\min\limits_{x\in A(\eta,\delta)} f(x)$ exists (and is equal to $2^{-n+1}$).
Hence, there is $\delta_0=\delta_0(\eta)>0$
with $\min\limits_{x\in A(\eta,\delta_0)} f(x)\geq 1.5\cdot 2^{-n}$.
Now, for any convex body $B$ with $B_\infty^n\subset B\subset (1+\delta_0)B_\infty^n$
and $x\in\partial B$ with $\|x-v\|_\infty\geq \eta$ for all $v\in\{-1,1\}^n$, we have
$x\in A(\eta,\delta_0)$, whence $\sigma(B,x)\geq f(x)\geq \min\limits_{y\in A(\eta,\delta_0)} f(y)\geq 1.5\cdot 2^{-n}$.
The statement follows.
\end{proof}

\begin{proof}[Proof of Claim~\ref{claim2}.]
Let us make a few preliminary observations. First, the Euclidean distance from the point
$p=(\frac{1}{\sqrt{n}},\frac{1}{\sqrt{n}},\ldots,\frac{1}{\sqrt{n}})$ to each coordinate vector $e_i$ is
$$\|p-e_i\|_2=\sqrt{2-\frac{2}{\sqrt{n}}}.$$
On the other hand, given two opposite points 
$p_1$ and $p_2=-p_1$ on the unit sphere $\Sph^{n-1}$, we have
$$\|p-p_1\|_2^2+\|p-p_2\|_2^2=4.$$
Therefore,
$$\max(\|p-p_1\|_2, \|p-p_2\|_2)\geq \sqrt{2}>\|p-e_1\|_2.$$
By continuity, there are $\tau=\tau(n)>0$ and $u=u(n)>0$ such that for any two points $p_1'$ and $p_2'$ on the unit sphere
such that the line passing through $p_1'$ and $p_2'$ is at the $\ell_\infty^n$--distance at most $\tau$ from the origin,
we have $\max(\|p-p_1'\|_2, \|p-p_2'\|_2)\geq \|p-e_1\|_2+u$.
Now, for each $\widetilde p\in \Sph^{n-1}$ denote by $K_{\widetilde p}$
the convex cone generated by vectors $e_1,e_2,\dots,e_n$ and $\widetilde p$.
It is clear that there is $\kappa=\kappa(n)>0$ such that
for any $\widetilde p\in \Sph^{n-1}$ with $\|\widetilde p-p\|_2\geq \|p-e_1\|_2+u$, we have
$\solid(K_{\widetilde p},0)\geq (1+\kappa)2^{-n}$.
Using a compactness argument, we infer that a slighly weaker inequality should hold in a small neighborhood of zero,
namely, there is $\tau'=\tau'(n)>0$ such that for any $z\in \tau'\,B_\infty^n$ and any point $\widetilde p$
in $\Sph^{n-1}$ with $\|\widetilde p-p\|_2\geq \|p-e_1\|_2+u$ we have that the solid angle of $z$
considered as the vertex of the convex cone generated by $e_i-z$ ($i=1,2,\dots,n$) and $\widetilde p-z$,
is at least $(1+\kappa/2)2^{-n}$.

Let us summarize. Let $\ell$ be an affine line in $\R^n$ and $z$ be a point on that line such that $z\in \min(\tau,\tau')B_\infty^n$
(with $\tau$ and $\tau'$ defined above). Further, let $p_1$ and $p_2$ be intersection points of this line with the unit sphere,
and take $p_j$ ($j\in\{1,2\}$) having the larger Euclidean distance from $p=(\frac{1}{\sqrt{n}},\frac{1}{\sqrt{n}},\ldots,\frac{1}{\sqrt{n}})$.
Let $K$ be the convex cone with vertex at $z$ generated by vectors $e_i-z$ ($i=1,2,\dots,n$) and $p_j-z$.
Then the solid angle of $z$ (w.r.t.\ $K$) is bounded fom below by $(1+\kappa/2)2^{-n}$.

The above assertion allows to complete the proof of the claim.
Let $x$ and $y$ be two distinct boundary points of a convex body $B\supset B_\infty^n$,
with $x-v,y-v\in \min(\tau,\tau')B_\infty^n$
for some vertex $v$ of $B_\infty^n$. Without loss of generality, we can assume that $v=(-1,-1,\ldots,-1)$.
Let $\ell$ be the line passing through $x-v$ and $y-v$, and let $p_1,p_2$ be intersection points of the
line with the unit sphere (let us assume for concreteness that $y-v$ lies in the interval joining $x-v$ and $p_2$).
Note that, since $B$ contains the cube and, in particular, points of the form
$v+e_i$, $i=1,2,\dots,n$, we get that $e_i-(x-v)$, $i=1,2,\dots,n$ lie in the closure of the illuminating cone $\IllumSet(B,x)$,
and the same is true for vector $p_2-(x-v)$. By analogy, the closure of the cone $\IllumSet(B,y)$
contains the vectors $e_i-(y-v)$ and $p_1-(y-v)$. Take $p_j$ having larger Euclidean distance to
$p=(\frac{1}{\sqrt{n}},\frac{1}{\sqrt{n}},\ldots,\frac{1}{\sqrt{n}})$.
If $j=1$ then, by the above reasoning, we have
$$\solid(B,y)\geq\solid\big(\conv\{p_1-(y-v),e_1-(y-v),\dots,e_n-(y-v)\},y-v\big)\geq (1+\kappa/2)2^{-n}.$$
Similarly, if $j=2$ then $\solid(B,x)\geq (1+\kappa/2)2^{-n}$. The claim follows.
\end{proof}

\subsection{Proof of Lemma~\ref{l: solid}}

\begin{Claim}\label{cl1}
Fix parameters $n>2$, $\theta\in(0,1)$, assume that $0<\varepsilon\leq \theta/2$
and $0<\delta\leq \varepsilon/7$, take a convex body $B$ in a $\star$-position with $\dist(B,B_\infty^n)\leq 1+\delta$,
and pick any set $S\in \Class_n(\varepsilon,\theta)$.
Let $v,v'$ be the distinguished pair of vertices with respect to $S$.
Then
for any $x\in\partial B$ with $x\notin (v+9\delta B_\infty^n)\cup(v'+9\delta B_\infty^n)$,
$x$ is illuminated by $S$.
\end{Claim}
\begin{proof}

Pick any point $x\in \partial B$, and let $w$ be a vertex of the cube
$B_{\infty}^n$ with the smallest distance to $x$; note that for all non-zero coordinates $x_j$
we have $\sign (w_j)=\sign (x_j)$. The choice of $w$ is not unique if $x$ has zero coordinates; in such situation
we pick any admissible vertex.

We will consider three cases.

\noindent\textit{Case 1: $w$ is adjacent neither to $v$, nor to $v'$.}
In this case, $S$ contains an element $s=-w+\varepsilon y$, for some $y\in B_{\infty}^n$.
By~\eqref{eq: cube incl ctr}, $B\subset (1+2\delta) B_{\infty}^n$, and we get for every $j=1,\dots,n$, and for any $a\in(0,1/2]$:
$$|(x+a(-w+\varepsilon y))_j|\leq 
|x_j-aw_j|+a\varepsilon
\leq
\max(1+2\delta-a,a)+a\varepsilon= 1+2\delta-a+a\varepsilon.$$
Selecting $a:=3\delta/(1-\varepsilon)$ we obtain that $x+as$ is in the interior of $B_{\infty}^n$.
Since $B_{\infty}^n\subset B$, it implies that $x+a s\in B\setminus\partial B$,
and therefore $x$ is illuminated by $s$.

\smallskip

\noindent\textit{Case 2: $w$ is adjacent either to $v$ or to $v'$, but $w\notin\{v,v'\}$.}
Let $i$ be the coordinate in which $v_i=v_i'\neq w_i$.
In this case, $S$ contains an element $s=-w+(1-\theta) w_i e_i +\varepsilon y$,
for some $y\in B_{\infty}^n$. As $B\subset (1+2\delta) B_{\infty}^n$,
we get for every $j\in[n]\setminus\{i\}$, and for any $a\in(0,1/2]$:
$$|(x+a(-w+(1-\theta) w_i e_i +\varepsilon y))_j|\leq 1+2\delta-a+a\varepsilon.$$
In addition,
$$|(x+a(-\theta w_i e_i+\varepsilon y))_i|\leq 1+2\delta-a\theta+a\varepsilon.$$
Selecting $a:=\frac{3\delta}{\theta-\varepsilon}$ and using the assumptions on $\delta$ and $\varepsilon$,
we obtain that $x+as\in B\setminus\partial B$. 
As before, it means that $x$ is illuminated by $s\in S$.

\smallskip

\noindent\textit{Case 3: $w$ is either $v$ or $v'$.}
By construction of $\Class_n(\varepsilon,\theta)$, $S$ contains an element
$$s\in -\prod_{j=1}^n \Big(\frac{v_j+v'_j}{2}\cdot[\theta,1]\Big)
+\varepsilon B_{\infty}^n.$$
Let $i$ be the coordinate in which $v$ and $v'$ differ (note that it is not the same $i$ as in the previous case).
Note that $\sign (x_j)=-\sign (s_j)$ for all non-zero coordinates $x_j$ with $j\in[n]\setminus \{i\}$.
We then get for every $j\in[n]\setminus\{i\}$ and for any $a\in(0,1/2]$:
$$|(x+as)_j|\leq 1+2\delta-a\theta+a\varepsilon.$$
In addition,
$$|(x+as)_i|\leq |x_i|+a\varepsilon.$$
Selecting $a:=\frac{3\delta}{\theta-\varepsilon}$, we obtain that $x+as\in B\setminus\partial B$
{\it unless} $|x_i|\geq 1-3\delta$.
Thus, $x$ is illuminated in the direction $s$ whenever $|x_i|< 1-3\delta$.
Now, assume that $|x_i|\geq 1-3\delta$; without loss of generality, $\sign (x_i)=\sign (v_i)$.
Similarly to Case 2, we note that $S$ contains directions
$s^k=-w^k+(1-\theta) w_k^k e_k +\varepsilon y^k$ ($k\in[n]\setminus\{i\}$),
where $y^k\in B_\infty^n$ and for each $k\in[n]\setminus\{i\}$, $w^k$ is the vertex of $\{-1,1\}^n$
adjacent to $v$ that differs from $v$ on the $k$-th coordinate.
Fix $k\in[n]\setminus\{i\}$.
For any $j\notin[n]\setminus\{k\}$ we have, just as in Case 2,
$$|(x+a(-w^k+(1-\theta) w_k^k e_k +\varepsilon y))_j|\leq 1+2\delta-a+a\varepsilon<1,$$
where the last inequality holds, for example, with $a:=3\delta/(1-\varepsilon)$.
Further, trivially
$$|(x+a(-w^k+(1-\theta) w_k^k e_k +\varepsilon y))_k|\leq |x_k|+a\theta+a\varepsilon,$$
and, with the last choice of $a$, the quantity is strictly less than $1$ whenever $|x_k|<1-9\delta$.
Thus, we get that $x$ is illuminated by one of the $n$ directions $\{s,s^k,\,k\in[n]\setminus\{i\}\}$,
whenever there is a coordinate of $x$ which is strictly less than $1-9\delta$ by absolute value.
The result follows.
\end{proof}

\begin{Claim}\label{clclmid}
For any $n>2$ and any $\varepsilon\in(0,1)$ there is $\eta'=\eta'(n,\varepsilon)\in(0,1)$ with the following property.
Let $x\in\Sph^{n-1}$ be a vector such that $x_j\geq \varepsilon$ and $x_k\leq -\varepsilon$
for some $j\neq k$. Then there is a vector $z\in\Sph^{n-1}$ with $z_i\geq\eta'$ for all $i\leq n$,
an such that $\langle z,x\rangle=0$.
\end{Claim}
\begin{proof}
Denote $a:=\sum_{i\neq j,k} x_i$.
First, let us define a vector $\widetilde z\in\R^n$ be setting $\widetilde z_i:=1$ for all $i\neq j,k$ and
$$\widetilde z_j:=\frac{\sqrt{n}}{x_j}>1;\quad \widetilde z_k:=\frac{\sqrt{n}+a}{-x_k}.$$
First, obviously $\langle z,x\rangle=0$ by the construction.
Further, by the Cauchy--Schwarts inequality, $a\geq -\sqrt{n}\sqrt{1-2\varepsilon^2}$ (where we used that $x$ is a unit vector).
Thus, all coordinates of $\widetilde z$ are greater than $\min(1,\sqrt{n}(1-\sqrt{1-2\varepsilon^2}))$.
It remains to choose $z:=\widetilde z/\|\widetilde z\|_2$.
\end{proof}

The next two claims can be verified with a usual compactness argument.
\begin{Claim}\label{newclaim}
For any $n>2$ and $\kappa\in(0,1)$ there is $\psi'=\psi'(n,\kappa)\in(0,1)$ with the following property.
Let $z^1,z^2,\dots,z^n$ be vectors in $\R^n$ such that $\|z^i-e_i\|_\infty\leq \psi'$.
Consider a convex cone
$$K:=\big\{x\in\R^n:\, \langle x, z^i\rangle\geq 0\;\;\mbox{for all }i\leq n\big\}.$$
Then
$$\solid(K,0)\leq (1+\kappa)\cdot 2^{-n}.$$
\end{Claim}
\begin{Claim}\label{newwclaim2}
For any $n>2$ and $\eta\in(0,1)$ there is $\psi''=\psi''(n,\eta)\in(0,1)$ with the following property.
Let $\widetilde z^1,\widetilde z^2,\dots,\widetilde z^n$ be vectors in $\R^n$ such that $\|\widetilde z^i-e_i\|_\infty\leq \psi''$.
Consider the convex cone $\widetilde K$ generated by vectors $\widetilde z^i$ ($i\leq n$).
Then any unit vector $f\in\R^n$ with $f_i\geq \eta$ for all $i=1,2,\dots,n$, lies in the interior of $\widetilde K$.
\end{Claim}

\begin{proof}[{Proof of Lemma~\ref{l: solid}}]
Fix parameters $n>2$ and $\kappa\in(0,1)$, and let $\eta'(\cdot,\cdot)$,
$\psi'(\cdot,\cdot)$ and $\psi''(\cdot,\cdot)$ be as in Claims~\ref{clclmid},
\ref{newclaim}
and~\ref{newwclaim2}, respectively.
Define $\psi:=\psi'(n,\kappa)/(4n)$ and $\theta_{\text{\tiny\ref{l: solid}}}:=\psi/8$.
Now, take any $\theta\in(0,\theta_{\text{\tiny\ref{l: solid}}}]$, and let
$$\varepsilon=\varepsilon_{\text{\tiny\ref{l: solid}}}:=\min(1/n^2,\theta\psi/4),\quad
\delta=\delta_{\text{\tiny\ref{l: solid}}}:=\min(\psi''(n,\eta'(n,\theta\psi/(8n)))/4.5,\varepsilon/7).$$
Further, take any element $S\in\Class_n(\varepsilon,\theta)$
and a convex body $B$ in a $\star$-position, with $B_\infty^n\subset B\subset (1+\delta)B_\infty^n$.
Our goal is to show that for any point $x$ on the boundary of $B$ not illuminated by $S$ we have
$\solid(B,x)\leq (1+\kappa)\cdot 2^{-n}$.

Let $v,v'$ be distinguished vertices of the standard cube $B_\infty^n$ with respect to $S$, and
assume that there is a point $x\in\partial B$ which is not illuminated by $S$.
By Claim~\ref{cl1}, $x\in (v+9\delta B_\infty^n)\cup(v'+9\delta B_\infty^n)$.
Without loss of generality, let us assume that $\|x-v\|_\infty\leq 9\delta$ and, moreover,
$v=(-1,-1,\dots,-1)$ and $v'=(1,-1,-1,\dots,-1)$.
For $i\leq n$, let $w^i$ be the vertex of $B_\infty^n$ adjacent to $v$, with $w_i^i\neq v_i$.
Then $w^i-x$ belongs to $\overline{\IllumSet(B,x)}$.
By our assumption, $w^i-x=2e_i+9\delta \widetilde y^i$ for some $\widetilde y^i\in B_{\infty}^n$,
so that
\begin{equation}\label{eq: aux -9hggs}
\mbox{$\overline{\IllumSet(B,x)}$ contains a convex cone generated by $e_i+4.5\delta \widetilde y^i$},\quad i\leq n.
\end{equation}

Next, we use the assumption that $x$ is not illuminated by any of the directions from $S$.
As $x$ is adjacent to a distinguished vertex, there are two types of illuminating directions from $S$ we will consider.
First, for any vertex $w$ of $B_\infty^n$ adjacent to $v$ but distinct from $v'$, $S$ contains a direction of the form
$$s^i:=-\sum_{j\neq i} w_j e_j-\theta w_i e_i+\varepsilon y^i,$$
for some $y^i\in B_\infty^n$, where $i\geq 2$ is the unique index such that $w_i\neq v_i=v_i'$.
A trivial computation gives
\begin{equation}\label{eq: aux agub}
s^i=(1,1,1,\ldots,1)-(1+\theta)e_i+\varepsilon y^i,\quad i=2,3,\ldots,n.
\end{equation}
Further, $S$ contains the distinguished direction $s^1$ such that
$$s^1\in-\prod_{j=1}^n \bigg(\frac{v_j+v_j'}{2}\cdot [\theta,1]\bigg)+\varepsilon B_\infty^n.$$
Since $x$ is not illuminated by any of the $s^i$'s, we have $s^i\notin \IllumSet(B,x)$, $i=1,2,\dots,n$.
The Hahn--Banach separation theorem implies that for any $i$ there is
an affine hyperplane $H_i$ passing through $x$ and parallel to $s^i$ but not intersecting the interior of $x+\IllumSet(B,x)$. 
Choose the unit normal vector $z^i$ to $H_i$ such that $\langle z^i,e_i\rangle\geq 0$.

First, consider the directions $s^i$ for $i\geq 2$.
Since $\langle s^i, z^i\rangle=0$ and in view of~\eqref{eq: aux agub} we have
$$
\sum_{j\neq i} z^i_j-\theta z^i_i+\varepsilon\langle y^i, z^i\rangle=0.
$$
Assume that there exist two coordinates $z^i_k$ and $z^i_\ell$ ($k\neq \ell$) of $z^i$ such that
$|z^i_k|,|z^i_\ell|\geq\psi$. Let us suppose for concreteness that $k\neq i$. Then
$|\sum_{j\neq i} z^i_j|\leq \theta+\varepsilon\sqrt{n}\leq \psi/4$
and $|z^i_k |\geq \psi$, whence there is $u\neq k$
such that $|z^i_u|\geq \psi/(2n)$ and the sign of $z^i_u$ is opposite to the sign of $z^i_k$.
Then, by Claim~\ref{clclmid}, there is a vector $f\in\Sph^{n-1}$ such that $\langle f,z^i\rangle=0$
and $f_p\geq \eta'(n,\psi/(2n))$ for all $p\leq n$. By Claim~\ref{newwclaim2},
applied to the convex cone $\widetilde K$ generated by vectors $\widetilde z^i:= e_i+4.5\delta \widetilde y^i$, and our choice of parameters,
we get that $f$ must belong to the interior of $\widetilde K$. On the other hand, the vectors $e_i+4.5\delta \widetilde y^i$
all belong to the closure of the cone $\eta(B,x)$ (see \eqref{eq: aux -9hggs}).
Hence, $f\in\IllumSet(B,x)$,
so that $H_i$ intersects the interior of $x+\IllumSet(B,x)$ and we come to contradiction. Thus, the vector $z^i$
has only one coordinate $z^i_b$ ($b\leq n$) with $|z^i_b|\geq \psi$ (note that automatically, $|z^i_b|> 1-n\psi$).
By the choice of $\psi$ and $\theta$, we necessarily have $b=i$, and so $\|z^i-e_i\|_\infty\leq n\psi<\psi'(n,\kappa)$.

Next, we apply a similar argument to direction $s^1$.
We have
$
\sum_{j=2}^n s_j^1 z^1_j+s_1^1 z^1_1=0
$,
where $\theta-\varepsilon\leq s_j^1\leq 1+\varepsilon$ for all $j\geq 2$ and $-\varepsilon\leq s_1^1\leq \varepsilon$.
Suppose that for some index $k\geq 2$ we have $|z^1_k|\geq\psi$.
Since $|s_1^1 z^1_1|\leq \varepsilon\leq \theta\psi/4$, we have
$|\sum_{j=2}^n s_j^1 z^1_j|\leq \theta\psi/4$ while $|s^1_k z^1_k|\geq \theta\psi/2$.
Hence, there is $\ell\geq 2$ ($\ell\neq k$) such that $|s^1_\ell z^1_\ell|\geq \theta\psi/(4n)$,
and the sign of $z^1_\ell$ is opposite to the sign of $z^1_k$.
By Claim~\ref{clclmid}, there is a vector $f\in\Sph^{n-1}$ such that $\langle f,z^i\rangle=0$ and
$f_p\geq \eta'(n,\theta\psi/(8n))$ for all $p\leq n$. An application of Claim~\ref{newwclaim2}
identical to the previous case, yields a contradiction. Thus, $|z^1_j|\leq\psi$ for all $j\geq 2$,
and $\|z^1-e_1\|_\infty\leq n\psi<\psi'(n,\kappa)$.

Observe that the cone $\eta(B,x)$ is contained inside the set $\{y\in\R^n:\,\langle y,z^i\rangle\geq 0\}$.
Finally, applying Claim~\ref{newclaim}, we get from the last observation
and our choice of parameters that $\solid(B,x)\leq (1+\kappa)\cdot 2^{-n}$,
completing the proof.
\end{proof}

\begin{rem}
We would like to point out that the proof of Lemma~\ref{l: solid}
requires that $\theta\gg\varepsilon$, and, this is the only place in the proof where this relation is used.
\end{rem}

\subsection{Proof of Lemma \ref{l: pseudo=cubic}}

Assumptions of the lemma imply that there exist a point
$x=(x_1,\dots,x_n)\in\Inter(B)$ with $|x_i|>1$ for at least one $i\in \{1,\dots,n\}$.
Without loss of generality, we can assume that $x_1>1$.
Then, by convexity of $B$ and the condition $B_\infty^n\subset B$,
there exists a point $y=(y_1,0,0,\dots,0)\in \Inter(B)$ with $y_1>1$.
Take two vertices $v:=(1,1,\dots,1)$ and $v':=(-1,1,\dots,1)$ of the unit cube, and define
$$p:=-\frac{v+v'}{2}+\varepsilon' e_1=(\varepsilon',-1,-1,\dots,-1),$$
where $\varepsilon':=\min(\varepsilon,(y_1-1)/2)$.
Note that $p$ illuminates $v'$ as a vertex of the 
cube, and hence it illuminates $v'$ viewed as a boundary point of $B$.
Additionally, as $v+p\in\Inter (B)$, we have that
$p$ illuminates $v$ as a boundary point of $B$.

Consider the collection of directions $S$ consisting of $p$,
of all the directions $-w$ where $w$ are the vertices of the unit cube not adjacent to $\{v,v'\}$,
and of the directions of the form
$-\sum_{j\neq i}w_j e_j- \theta w_i e_i$, for the vertices $w$ adjacent to either $v$ or $v'$
and different from $v,v'$ in $i$-th coordinate ($i=2,3,\dots,n$).
Note that $S$ belongs to the class $\Class_n(\varepsilon,\theta)$ and that each point of $\{-1,1\}^n$
is illuminated by $S$.

\subsection{Proof of the Lemma~\ref{l: pseudo not cubic}}

We begin with the following elementary claim. 

\begin{Claim}\label{claim1lemma7}
Consider the collection $\{w^i\}_{i=1}^m$ of $m$ vertices of the discrete cube $\{-1,1\}^m$ in $\R^m$,
where for each $i=1,\dots, m$, all coordinates of $w^i$ except for the $i$-th are $+1$,
and the $i$-th coordinate is $-1$.
Then the $(m-1)$--dimensional affine linear span of $w^1,w^2\ldots,w^m$ contains no other vertices of the cube except $w^1,\ldots,w^m$.
\end{Claim}
\begin{proof} Let $x$ be a point in the affine linear span of 
$w^1,\ldots,w^m$, i.e.\ $x=\sum \alpha_i w^i$ for some coefficients $\alpha_i\in\R$ so that $\sum \alpha_i=1$.
By the definition of $w^i$'s, we have $x=(1-2\alpha_1, 1-2\alpha_2,\dots,1-2\alpha_m)$.
Assume that $x\in\{-1,1\}^m$.
Then $|1-2\alpha_i|=1$ for all $i\leq m$, whence all $\alpha_i$'s are equal to either $0$ or $1$.
As their sum is one, that means that all the $\alpha_i$'s except for one are equal to zero.
Hence, $x$ coincides with one of the points $w^i$'s, which proves the claim.
\end{proof}
\begin{rem}\label{rem1lemma7}
The above claim implies that for any point $w\in\{-1,1\}^m\setminus\{w^1,\ldots,w^m\}$ the
simplex $\Delta_w$ with vertices $\{w,w^1,w^2,\ldots,w^m\}$ is non-degenerate. In particular, it follows
that there is a real number $u=u(m)>0$ depending only on $m$ such that 
for all $w\in\{-1,1\}^m\setminus\{w^1,\ldots,w^m\}$
the average $\frac{1}{m+1}(w+w^1+\ldots+w^m)$
is at the Euclidean distance at least $u$ from any supporting hyperplane for $\Delta_w$.
\end{rem}

\bigskip

Let $0<r\leq 1/4$ be a parameter, and let $P$ be a convex polytope in $\R^n$ such that
\begin{equation}\label{eq: aux -49u3}
\mbox{$P$ has $2^n$ vertices, and
$\forall\;v\in\{-1,1\}^n\;\;\exists\;\mbox{ a vertex }\widetilde v$ of $P$ with $\widetilde v-v\in r\,B_\infty^n$.}
\end{equation}
Let us make an immediate elementary observation that will be useful later:
\begin{Claim}\label{claimthetaLemma7}
For any $n>2$ and $\theta\in(0,1)$ there is $\widetilde r=\widetilde r(n,\theta)\in(0,1/4]$ with the following property.
Let $P$ be a convex polytope in $\R^n$ satisfying \eqref{eq: aux -49u3} with parameter $r\leq \widetilde r$,
and let $w$ be any vertex of $P$. Then for any $k\leq n$ the vector $(1-\theta)w_k e_k$ belongs to $\Inter(P)$.
\end{Claim}

\medskip

Let $P$ be as in \eqref{eq: aux -49u3}.
For any $i\leq n$, let $V_i^+=V_i^+(P)$ (resp., $V_i^-=V_i^-(P)$) be the set of vertices of $P$ with positive
(resp., negative) $i$-th coordinates.
Further, for any $i\leq n$ we introduce special collections $\mathcal W_i^+=\mathcal W_i^+(P)\subset V_i^+$
and $\mathcal W_i^-=\mathcal W_i^-(P)\subset V_i^-$,
where $\mathcal W_i^+$ is the set of $n-1$ vertices of $P$ (from $V_i^+$)
each having exactly $n-1$ positive coordinates,
and $\mathcal W_i^-$ is the set of $n-1$ vertices of $P$
having exactly $2$ negative coordinates (one of them the $i$-th).
Note that, when $P$ is the standard cube $[-1,1]^n$,
the sets $\mathcal W_i^+$ and $\mathcal W_i^-$, with the $i$-th components
of the vertices removed, directly correspond to the vertex sets from Claim~\ref{claim1lemma7}, with $m=n-1$.

The next statement obviously holds for the standard cube (see Remark~\ref{rem1lemma7}).
Its extension for very small perturbations of the cube follows by continuity. We omit the proof.
\begin{Claim}\label{claim1cLemma7}
For each $n\geq 3$ there is $r''=r''(n)\in(0,1/4]$ with the following property:
Let $P$ be a polytope in $\R^n$ satisfying \eqref{eq: aux -49u3}
with parameter $r\leq r''$, let $i\leq n$, and let $w^+\in V_i^+\setminus \mathcal W_i^+$ and
$w^-\in V_i^-\setminus \mathcal W_i^-$ be two points with $\sign(w^+_j)=\sign(w^-_j)$ for all $j\neq i$.
Denote by $H$ the affine linear span of $\mathcal W_i^-\cup\{w^-\}$.
Further, let $\widetilde w^+\in \mathcal W_i^+\cup\{w^+\}$ and $\widetilde w^-\in \mathcal W_i^-\cup\{w^-\}$ be two points
with $\sign(\widetilde w^+_j)=\sign(\widetilde w^-_j)$ for all $j\neq i$.
Set $c^+:=\frac{1}{n}(w^++\sum_{w\in \mathcal W_i^+}w)$ and
let $q$ be the point in $H$ such that $q-c^+$ is parallel to $\widetilde w^- -\widetilde w^+$. Then necessarily $q$
belongs to the interior of the simplex $\conv(W_i^-\cup\{w^-\})$.
\end{Claim}

In the next claim, we give a sufficient condition for $P$ to be a parallelotope.

\begin{Claim}\label{claim2lemma7}
Suppose that a polytope $P$ in $\R^n$
satisfies \eqref{eq: aux -49u3} with $r\leq r''$, where $r''$ is given by 
Claim~\ref{claim1cLemma7}.
Further, suppose that for every $i\leq n$ and every pair of vertices $(v^+,v^-)\in (V_i^+\setminus \mathcal W_i^+)\times
(V_i^-\setminus \mathcal W_i^-)$
such that $\sign(v^+_j)=\sign(v^-_j)$ for all $j\neq i$,
we have that the affine spans of $\mathcal W_i^+\cup\{v^+\}$ and $\mathcal W_i^-\cup\{v^-\}$ are parallel
and, moreover, both spans are supporting hyperplanes for $P$. Then necessarily $P$ is a parallelotope.
\end{Claim}
\begin{proof}
Fix any $i\leq n$, let $(v^+,v^-)$ be a pair as above,
and let $z\in \Sph^{n-1}$ be the (unique) vector such that for some number $\alpha<0$ we have
$\langle z,w\rangle=\alpha$ for all $w\in \mathcal W_i^-\cup\{v^-\}$.
By the assumption of the claim, the affine spans $H_1$ and $H_2$
of $\mathcal W_i^+\cup\{v^+\}$ and $\mathcal W_i^-\cup\{v^-\}$ are parallel,
hence there is a number $\beta>0$ such that $\langle z,w\rangle=\beta$ for all $w\in\mathcal W_i^+\cup\{v^+\}$.
Next, choose any vertex $w^+\in V^+\setminus (\mathcal W_i^+\cup \{v^+\})$
and the vertex $w^-\in V^-\setminus (\mathcal W_i^-\cup \{v^-\})$ such that 
$\sign(w^+_j)=\sign(w^-_j)$ for all $j\neq i$.
Again, the affine spans $H_1'$ and $H_2'$ of $\mathcal W_i^+\cup\{w^+\}$ and $\mathcal W_i^-\cup\{w^-\}$ are parallel.
On the other hand, by the fact that $H_1,H_2$ are both supporting hyperplanes for $P$,
we get that $\langle z,w^+\rangle\leq \beta$ and $\langle z,w^-\rangle\geq \alpha$.

Assume for a moment that at least one of the last two inequalities is strict. Note that in this case necessarily
both inequalities are strict, i.e.\ $\langle z,w^+\rangle< \beta$ and $\langle z,w^-\rangle> \alpha$
(otherwise, we would get an immediate contradiction to the fact that $H_1$ and $H_2$,
as well as $H_1'$ and $H_2'$ are parallel). Next, let
$c^+$ and $c^-$ be arithmetic means of points in $\mathcal W_i^+\cup\{w^+\}$ and $\mathcal W_i^-\cup\{w^-\}$, respectively,
fix any pair $(\widetilde w^+,\widetilde w^-)\in \mathcal W_i^+\times \mathcal W_i^-$ such that
$\sign(\widetilde w^+_j)=\sign(\widetilde w^-_j)$ for all $j\neq i$,
and define $q$ as the point in $H_2'$ such that $q-c^+$ is parallel to $\widetilde w^- -\widetilde w^+$.
By Claim~\ref{claim1cLemma7}, the point $q$ belongs to the interior of the simplex $\conv(\mathcal W_i^-\cup\{w^-\})$.
This gives
$\langle z,q\rangle>\alpha$. Let us summarize: we have obtained four points $\widetilde w^+$, $c^+$, $q$,
$\widetilde w^-$ forming a parallelogram, but
\begin{align*}
\langle z,\widetilde w^+\rangle&=\beta\geq\langle z, c^+\rangle;\\
\langle z,\widetilde w^-\rangle&=\alpha<\langle z,q\rangle,
\end{align*}
which is impossible.
Thus, necessarily $\langle z,w^+\rangle=\beta$ and $\langle z,w^- \rangle=\alpha$,
implying that $H_1$ coincides with $H_1'$ and $H_2$ coincides with $H_2'$.

Repeating the above argument for all vertices in $V_i^+\setminus \mathcal W_i^+$
and $V_i^-\setminus \mathcal W_i^-$, we get that there is a single facet of $P$ containing all vertices $V_i^+$,
and the same holds for $V_i^-$.
Since this condition holds for any $i\leq n$, we get that $P$ is generated by $n$ pairs of parallel hyperplanes,
so $P$ is a parallelotope.
\end{proof}

\begin{proof}[Proof of Lemma \ref{l: pseudo not cubic}]
Fix parameters $n>2$ and $\varepsilon>0$, and take
$\theta_{\text{\tiny\ref{l: pseudo not cubic}}}:=\frac{1}{4n}$.
Let $0<\theta\leq \theta_{\text{\tiny\ref{l: pseudo not cubic}}}$,
and let $P$ be a polytope in $\R^n$ satisfying
condition~\eqref{eq: aux -49u3} for
$r=r_{\text{\tiny\ref{l: pseudo not cubic}}}:=\min(\widetilde r,r'',\theta/5,\varepsilon/4)$, where $\widetilde r,r''$
are given by Claims~\ref{claimthetaLemma7} 
and~\ref{claim1cLemma7}.
Assume further that $\dist(P,B_\infty^n)\neq 1$.
We want to show that $P$ can be illuminated by some $S\in\Class_n(\varepsilon,\theta)$.
By Claim~\ref{claim2lemma7},
there is $i\leq n$ and two sets of vertices $\mathcal{W}_i^+\cup\{v^+\}$
and $\mathcal{W}_i^-\cup\{v^-\}$ (where $\sign(v^+_j)=\sign(v^-_j)$ for all $j\neq i$), such that one of the following two conditions holds.

\noindent\textit{Case 1.}
The affine span of $\mathcal{W}_i^+\cup\{v^+\}$ is parallel to the affine span of $\mathcal{W}_i^-\cup\{v^-\}$,
but at least one of the spans is not a supporting hyperplane for $P$.
Without loss of generality, assume that the simplex with vertices $\mathcal{W}_i^+\cup\{v^+\}$ is not a part of a facet,
whence the point $c^+:=\frac{1}{n}(v^++\sum_{w\in \mathcal{W}_i^+} w)$ belongs to the interior of $P$.
Consider the unit vector $z$ orthogonal to both affine spans of $\mathcal{W}_i^+\cup\{v^+\}$
and $\mathcal{W}_i^-\cup\{v^-\}$, 
chosen so that $\langle z, v^-\rangle< \langle z, v^+\rangle$.
Then, for all sufficiently small $\xi>0$, the vector $p=c^+-v^+ +\xi z$
illuminates both $v^+$ and $v^-$.
Indeed, illumination of $v^+$ is obvious. As for $v^-$, note that, by Claim~\ref{claim1cLemma7},
there is a point $q$ in the interior of the $(n-1)$--simplex $\conv(\mathcal{W}_i^-\cup\{v^-\})$ such that
$q-c^+$ is parallel to $v^- -v^+$. This means that $q-v^-=c^+-v^+$,
whence $v^- + p$ belongs to the interior of $P$ for small enough $\xi$.

We construct the illuminating set $\mathcal S$ for $P$ as follows.
Take any vertex $w$ of $P$ and let $\widetilde w$ be the corresponding vertex of $B_\infty^n$
(i.e.\ such that $w-\widetilde w\in r B_\infty^n$).
Further, denote by $v$ and $v'$ the vertices of $B_\infty^n$ corresponding to $v^+$ and $v^-$.
If $\widetilde w$ is not adjacent to $\{v,v'\}$ then we add to $\mathcal S$ the direction $-w$
(observe that $-w\in -\widetilde w+\varepsilon B_\infty^n$ and that $-w$ illuminates $w$).
Next, if $\widetilde w$ is adjacent to $v$ or $v'$ but does not belong to $\{v,v'\}$ then
we pick the direction $-w+(1-\theta) w_k e_k$,
where $k\neq i$ is the unique index such that $v_k=v'_k\neq \widetilde w_k$.
Observe that, in view of Claim~\ref{claimthetaLemma7}, the direction $-w+(1-\theta) w_k e_k$
illuminates $w$. On the other hand,
$-w+(1-\theta) w_k e_k\in -\widetilde w+(1-\theta)\widetilde w_k e_k+ \varepsilon B_\infty^n$.
Finally, if $\widetilde w$ coincides with either $v$ or $v'$ then
we consider the direction $p$ constructed above.
Observe that $|p_i|\leq 3r$ (assuming $\xi$ is small). Further, for any $j\neq i$ we have
$|c^+_j-v^+_j+\xi z_j|\geq \frac{1}{n}-2r-\xi\geq 2\theta$
and $\sign(p_j)=\sign(c^+_j-v^+_j +\xi z_j)=-\sign(v_j)$, where we used the definition of the set $\mathcal{W}_i^+$
and the assumption that $\xi$ is small.
Hence,
$$\frac{1}{2}p\in -\prod_{j=1}^n \bigg(\frac{v_j+v_j'}{2}\cdot [\theta,1]\bigg)+2r B_\infty^n.$$
We add $\frac{1}{2}p$ to $S$ as the distinguished direction.
By the construction of $S$, it illuminates all vertices of $P$,
and belongs to the class $\Class_n(\varepsilon,\theta)$.

\noindent\textit{Case 2.}
Suppose that the affine span of $\mathcal{W}_i^+\cup\{v^+\}$ is not parallel to the affine span of
$\mathcal{W}_i^-\cup\{v^-\}$.
Choose a vertex $\widetilde v^+\in\mathcal{W}_i^+\cup\{v^+\}$ with the smallest distance to the affine span $H$ of
$\mathcal{W}_i^-\cup\{v^-\}$, and fix the unit vector $z$ orthogonal to $H$ and such that
$\langle z,\widetilde v^+\rangle>0$. Let $\widetilde v^-$ be the vertex in $\mathcal{W}_i^-\cup\{v^-\}$
corresponding to $\widetilde v^+$ (i.e.\ $\sign(\widetilde v^+_j)=\sign(\widetilde v^-_j)$ for all $j\neq i$),
let $c^+$ be the average of points in $\mathcal{W}_i^+\cup\{v^+\}$, and consider the vector
$\widetilde p:= c^+-\widetilde v^+ -\xi z$, where $\xi>0$ is a small parameter. 
Observe that for small $\xi$, $\widetilde p$ illuminates $\widetilde v^+$ by construction.
Further, by Claim~\ref{claim1cLemma7}, there is a point $q$ in the interior of the simplex $\mathcal{W}_i^-\cup\{v^-\}$,
such that $q- c^+$ is parallel to $\widetilde v^- - \widetilde v^+$.
Since the span of $\mathcal{W}_i^+\cup\{v^+\}$ is not parallel to $H$ and by our choice of $\widetilde v^+$ and $z$,
we have that $\langle c^+-\widetilde v^+,z\rangle>0$ while $\langle q- \widetilde v^-,z\rangle=0$.
Hence, the point $\widetilde v^- + (c^+-\widetilde v^+)$ lies in the interior of the interval joining $c^+$ and $q$
i.e.\ in the interior of $P$, and so
$\widetilde p$ illuminates $\widetilde v^-$ provided that $\xi$ is sufficiently small.
We will construct an illuminating set $S$ for $P$ taking $\frac{1}{2}\widetilde p$ as the distinguished direction.
The rest of the argument is very similar to the first case, and we omit the details.

\end{proof}

\end{document}